\documentclass[12pt]{amsart}
\usepackage[utf8]{inputenc}
\usepackage{enumitem}
\usepackage{verbatim}
\usepackage{amsmath, amsthm,amssymb,bbm,enumerate,mathrsfs,mathtools}

\title{Edge-disjoint Linkage in Infinite Graphs}
\author[A.~Assem]{Amena Assem}
\address{University of Waterloo and York University, ON, Canada}
\email{a36mahmo@uwaterloo.ca, amnassem@yorku.ca}

\author[R.B.~ Richter]{R. Bruce Richter}
\address{University of Waterloo, ON, Canada}
\email{brichter@uwaterloo.ca}

\date{}

\keywords{infinite graph, linkage, edge-connectivity, edge-disjoint, lifting.
\\
Reserach supported by NSERC \# 50503-10940-500}

\usepackage{xcolor} 	
\usepackage[unicode]{hyperref}
\hypersetup{
	colorlinks,
    linkcolor={red!60!black},
    citecolor={green!60!black},
    urlcolor={blue!60!black},
}
\usepackage[abbrev, msc-links]{amsrefs}

\usepackage{geometry}
\theoremstyle{plain}
\newtheorem{theorem}{Theorem}[section]

\newtheorem{proposition}[theorem]{Proposition}

\newtheorem{corollary}[theorem]{Corollary}
\newtheorem{lemma}[theorem]{Lemma}

\newtheorem{conjecture}[theorem]{Conjecture}

\theoremstyle{definition}

\def\eop{\hfill{{\rule[0ex]{8.25pt}{8.25pt}}}}

\begin{document}

\begin{abstract}
    In 1980, Thomassen stated his \emph{weak linkage} conjecture: for an odd positive integer $k$, if a graph $G$ is $k$-edge-connected, then, for any collection of $k$ pairs of vertices $\{s_1,t_1\}$, \dots, $\{s_k,t_k\}$ in $G$, not necessarily distinct, there are pairwise edge-disjoint paths $P_1,\dots,P_k$ in $G$, with $P_i$ joining $s_i$ and $t_i$.  In 1991, Huck proved that the conclusion holds if $G$ is finite and $(k+1)$-edge-connected.  We prove that Huck's theorem holds also for all infinite graphs, extending and improving a result of Ok, Richter and Thomassen for 1-ended,  locally finite graphs.
    
A novel key tool in the proof is the \emph{linking fan proposition} proved in Section 3. To show the potential and usefulness of this proposition in other contexts, we apply it in the last section to prove a new result, similar to a result of Thomassen, on the existence of $2k$-edge-connected finite immersions in $(2k+1)$-edge-connected infinite graphs. We then use this to prove that an edge-connectivity of $2k+1$ is sufficient for infinite graphs to admit a $k$-arc-connected orientation. This is only within $1$ of the longstanding conjecture of Nash-Williams from 1960 that an edge-connectivity of $2k$ should be enough.

\end{abstract}

\maketitle

\section{Introduction}

All graphs in this paper may have parallel edges but no loops. In 1980, Thomassen introduced the \emph{weak linkage conjecture} \cite[Conjecture\ 1]{thomassen1980linkages}, which we rephrase here as follows.

\begin{conjecture}\label{cj:weakLinkage}  If $k$ is an odd positive integer and $G$ is a $k$-edge-connected graph, then, for any collection of $k$ pairs of vertices $\{s_1,t_1\}$, \dots, $\{s_k,t_k\}$ in $G$, there are $k$ pairwise edge-disjoint paths $P_1,\dots,P_k$ in $G$, with $P_i$ joining $s_i$ and $t_i$.
\end{conjecture}

In Conjecture \ref{cj:weakLinkage}, for each $i = 1, 2, \dots, k$, $s_i \neq t_i$, but otherwise there are no constraints on
the vertices $s_i$ and $t_i$. In particular, the special case in which all the $s_i$ are equal and all the $t_i$ are equal
is Menger's Theorem.

A set $\{P_1,\dots, P_k\}$ of paths as in Conjecture \ref{cj:weakLinkage} is a \emph{weak $k$-linkage} in $G$ and, if every set of $k$ pairs of vertices has a weak $k$-linkage, then $G$ is \emph{weakly $k$-linked}.  Thus, for odd positive integers $k$, Conjecture \ref{cj:weakLinkage} asserts that a $k$-edge-connected graph is weakly $k$-linked. If we denote the set of pairs by $T$, we sometimes call the linkage a \emph{$T$-linkage}.  

Thomassen gave the example of a $2k$-cycle $(s_1,\dots,s_k,t_1,\dots,t_k)$ with each edge replaced by $k/2$ parallel edges to show that, if $k$ is even, $k$-edge-connectivity is not sufficient. He conjectured that $(k+1)$-edge-connectivity suffices for even $k$. This second conjecture is obviously a consequence of Conjecture \ref{cj:weakLinkage}.

For finite graphs, after several intermediate advances leading to Okamura proving that $(4k/3)$-edge-connectivity suffices \cite{okamura}, in 1991 Huck \cite{huck1991} proved the following; the best result to date.

\begin{theorem}\label{th:huck}(Huck's Theorem) If $k$ is an odd positive integer and $G$ is a finite $(k+1)$-edge-connected graph, then $G$ is weakly $k$-linked.
\end{theorem}

The main point of this work is to prove that Huck's theorem \ref{th:huck} extends to infinite graphs. This improves on the following theorem by Ok, Richter, and Thomassen.

\begin{theorem}\label{ORT}\cite[Theorem 1.3]{ORT2016linkages}
Let $k$ be an odd positive integer. A $(k+2)$-edge-connected, $1$-ended, locally finite graph is weakly $k$-linked.  
\end{theorem}

 We present our main result in the following theorem. This theorem does not give an independent proof of Huck's theorem, as his result is used in proving ours.

\begin{theorem}\label{th:huckInfinite} (Main Result)  If $k$ is an odd positive integer and $G$ is an infinite $(k+1)$-edge-connected graph, then $G$ is weakly $k$-linked.
\end{theorem}

There are two main components of the proof. One is a refinement of a lifting technique used by Thomassen in \cite{thomassen2016orientations} and Ok, Richter and Thomassen in \cite{ORT2016linkages}, and the other is the \emph{linking fan proposition}. These are needed for the locally finite case and their details are presented in Sections \ref{sec:lifting} and \ref{sec:fan} respectively.  In Section \ref{sec:proof}, we provide an explanation of how to apply Thomassen's reduction methods \cite{thomassen2016orientations} to reduce the case of a general (possibly uncountable) infinite graph to the locally finite case.  The now-simple proof of Theorem \ref{th:huckInfinite} is also in Section \ref{sec:proof}. Finally, Section \ref{sec:orientations} applies the same methods to find highly connected immersions and orientations in infinite graphs. New progress towards the orientation conjecture of Nash-Williams is presented in Theorem \ref{orientation theorem} where we prove that $(2k+1)$-edge-connectivity implies the existence of a $k$-arc-connected orientation in infinite graphs.

\section{Lifting in locally finite graphs}\label{sec:lifting}

In this section, we present the lifting result that we need for our proof of Theorem \ref{th:huckInfinite}.  If $sx$ and $sy$ are two edges in a finite graph $G$, then the \emph{lift of $sx$ and $sy$} is the graph $G_{sx,sy}$ obtained from $G$ by deleting $sx$ and $sy$ and adding (a possibly duplicate) edge $xy$. If $x=y$, we delete the resulting loop (so in this case lifting is just deleting two parallel edges). In many articles, \emph{split} is used in place of \emph{lift}. Some of the earliest results on lifting were proved in the 1970s by Lov\'asz \cite{lovasz1976eulerian} (who used the term \emph{splitting}), particularly for Eulerian graphs, and Mader \cite{mader1978reduction} (who used the term \emph{lifting} and considered general graphs). Later, more was proved by Frank \cite{frank}, and others. 

For a finite graph $G$, a positive integer $k$, and a vertex $s$ of $G$, $G$ is \emph{$(s,k)$-edge-connected} if, for any two vertices $u$ and $v$ of $G-s$, there are $k$ pairwise edge-disjoint paths between $u$ and $v$ in $G$ (that can possibly go through $s$). It is useful later to note that if $G$ is $(s,k)$-edge-connected and $\deg(s)\geq k$, then $G$ is $k$-edge-connected. 

Ok, Richter, and Thomassen \cite{ORT2016linkages} introduced the \emph{k-lifting graph} for an $(s,k)$-edge-connected graph $G$ to be the graph $L(G,s,k)$ whose vertex-set is the set of edges incident with $s$ and where two edges $sx$ and $sy$ are adjacent if $G_{sx,sy}$ is $(s,k)$-edge-connected as $G$ is. In that case we say that the pair $sx$ and $sy$ is $k$-\emph{liftable} or simply \emph{liftable} if the connectivity is understood from the context. The \emph{bad graph} defined by Thomassen in \cite{thomassen2016orientations} and the \emph{non-admissibility graph} introduced by Jord\'an in \cite{Jordan} are in essence the same as the complement of the lifting graph. An important remark that we will use in our proof is that if a pair of edges $su$ and $sv$ is not liftable in $G$, then it is also not liftable in $G_{sx,sy}$, that is, $L(G_{sx,sy},s,k)$ is a subgraph of $L(G,s,k)$, for any neighbours $x$ and $y$ of $s$. 

Thomassen \cite{thomassen2016orientations} proved the following.

\begin{lemma}\label{lm:kDegBothEven} \cite[Theorem 2]{thomassen2016orientations}  Let $k$ be a positive even integer and let $s$ be a vertex of an $(s,k)$-edge-connected Eulerian finite graph $G$.  Then the complement of $L(G,s,k)$ is disconnected. \hfill\eop
\end{lemma}

 This lemma was used by Thomassen in the same paper to show that if $G$ is an infinite $8k$-edge-connected graph, then $G$ has a $k$-arc-connected orientation.  This was the first result towards Nash-Williams' claim that $2k$-edge-connectivity suffices for the existence of a $k$-arc-connected orientation in infinite graphs (which Nash-Williams proved is sufficient for finite graphs \cite{nash1960orientations}).

 The following stronger statement on the $k$-lifting graph of a general finite graph, not necessarily Eulerian, was later proved by Ok, Richter, and Thomassen. The same result was independently proved before that by Jord\'an in \cite[Theorem 3.2]{Jordan} for the purpose of applications in connectivity augmentation and network optimization problems.

 \begin{lemma} \cite[Theorem 1.2]{ORT2016linkages}
 Let $k$ be a positive even integer and let $s$ be a vertex of an $(s,k)$-edge-connected finite graph $G$. If $\deg(s)$ is even, then the complement of $L(G,s,k)$ is disconnected. \hfill\eop
 \end{lemma}

This latter result was used by Ok, Richter, and Thomassen in \cite{ORT2016linkages} to prove their Theorem \ref{ORT} stated above. We need a similar understanding of the $k$-lifting graph when $\deg(s)$ is odd. We have the following result for even connectivity $k$ which covers both cases of the parity of $\deg(s)$ (except for $\deg(s)=3$ which needs a special treatment as shown in \cite{Assem2022lifting}). This is the special case of \cite[Theorem~3.3]{Max2023orientation} with $A=V\setminus \{s\}$, which is stated again in the same paper, and proved after providing the necessary lifting graph background under \cite[Theorem~4.8]{Max2023orientation}. It also follows directly from \cite[Theorem~1.5 (ii)]{Assem2022lifting} for $\deg(s)> 4$ and \cite[Proposition~3.4]{ORT2016linkages} for $\deg(s)=4$.

A complete bipartite graph is \emph{balanced} if its two parts have the same size.

\begin{lemma}\label{lm:amena}  
Let $k$ be a positive integer and $G$ an $(s,k)$-edge-connected finite graph such that $\deg(s)\geq 4$. If $k$ is even, then

\begin{itemize}

\item either the complement of $L(G,s,k)$ is disconnected, or 

\item $L(G,s,k)$ is the union of an isolated vertex and a balanced complete bipartite graph.
\hfill\eop
\end{itemize} 
\end{lemma}

Note that in the latter case the degree of $s$ is odd and the union of the isolated vertex with each side of the bipartite graph gives two independent sets (containing no adjacencies) in $L(G,s,k)$, of size $(\deg(s)+1)/2$ each, intersecting in the isolated vertex. We will also need the following lemma which is implied by point (6) in \cite[Lemma 3.4]{Assem2022lifting}. We include a proof here for completeness.

For this proof it is important to note that if the end-vertices $x$ and $y$ of two edges incident with $s$ are contained in a set $A$ such that $|\delta(A)|\leq k+1$ (where $\delta(A)$ denotes the \emph{boundary}, namely the set of edges with exactly one end-vertex in $A$), $s\notin A$, and $\overline{A\cup \{s\}} \neq \emptyset$, then lifting $sx$ and $sy$ results in a cut of size at most $k-1$ in $G$ which separates two non-$s$ vertices, meaning this pair of edges is not $k$-liftable. Conversely, in \cite[Theorem 1.1]{ORT2016linkages}, Ok, Richter, and Thomassen showed, under certian conditions which happen to be satisfied in Lemma \ref{lm:amena} above, that for every independent set in $L(G,s,k)$ (a set of pairwise non-liftable edges), such a set $A$, called a \emph{dangerous} set, exists in $G$ and contains all the non-$s$ end-vertices of the edges from the independent set. The conditions that need to be satisfied are that $\deg(s)\neq 3$, and $s$ is not incident with a cut-edge, which is the case above because $k\geq 2$ (as it is even) and because $G$ is $(s,k)$-edge-connected. We will show that these conditions are also met in the lemma below. Note that the structure of an isolated vertex plus a balanced complete bipartite graph for $L(G,s,k)$ may occur regardless of the parity of $k$ \cite[Figure 6]{Assem2022lifting}.

To prove the next lemma, we will make use of the following equation for two intersecting sets $A_1$ and $A_2$, which can be verified by simple counting, where $\delta(A_1:A_2)$ denotes the set of edges between $A_1$ and $A_2$.

\begin{equation} \label{intersection of two cuts}
\begin{split}
2 \bigg[ &\big|\delta(A_1)\big| +\big |\delta(A_2)\big|- \Big (\big|\delta(A_1 \cap A_2: \overline{A_1 \cup A_2})\big|+\big|\delta(A_2\setminus A_1: A_1\setminus A_2)\big|\Big) \bigg] 
\\
&= 
\big|\delta(A_1\cap A_2)\big|+\big|\delta(A_2\setminus A_1)\big|+ \big|\delta(A_1\setminus A_2)\big|+\big|\delta(\overline{A_1\cup A_2})\big|.
\end{split}
\end{equation}

\begin{lemma}\cite[Lemma 3.4]{Assem2022lifting} \label{degree at most k+2}
Let $k\geq 2$ be an integer and $G$ an $(s,k)$-edge-connected finite graph. If $L(G,s,k)$ is the union of an isolated vertex and a balanced complete bipartite graph, then $\deg(s)\leq k+2$.
\end{lemma}

\begin{proof}
Let $I_1$ and $I_2$ be the two independent sets in $L(G,s,k)$, each of size $(\deg(s)+1)/2$ and consisting of the union of the isolated vertex with one side of the bipartite graph. Then, $|I_1\cap I_2|=1$. Given the structure of $L(G,s,k)$, $\deg(s)$ is at least $3$, the isolated vertex plus one vertex on each side. If $\deg(s)=3$, then $\deg(s)\leq k+2$ as $k\geq 2$. So we may assume that $\deg(s)\geq 4$. Then by \cite[Theorem 1.1]{ORT2016linkages}, there are two sets $A_1$ and $A_2$ such that $|\delta_G(A_1)|, |\delta_G(A_2)|\leq k+1$, $s\notin A_1\cup A_2$, and for $i=1,2$, $\overline{A_i\cup \{s\}} \neq \emptyset$ and $A_i$ contains the non-$s$ end-vertices of the edges in $I_i$.

The only neighbour of $s$ contained in $A_1\cap A_2$ is the non-$s$ end-vertex of the edge $e$ of $G$ that is the isolated vertex in $L(G,s,k)$. To see this, note that every edge in $I_1\setminus \{e\}$ is liftable with every edge in $I_2\setminus \{e\}$ as $L(G,s,k) - \{e\}$ is complete bipartite. Thus an edge from $I_1\setminus \{e\}$ cannot have its non-$s$ end-vertex in $A_2$ as this would imply it is not liftable with the edges of $I_2\setminus \{e\}$ (since $|\delta_G(A_2)|\leq k+1$). Similarly, an edge from $I_2\setminus \{e\}$ cannot have its non-$s$ end-vertex in $A_1$.

 We will show that there are no vertices outside $A_1\cup A_2$ in $G$ except $s$. For this, we will apply Equation \ref{intersection of two cuts} in $G-s$. Suppose for a contradiction that $\overline{A_1\cup A_2\cup \{s\}}\neq \emptyset$. Then since the edges of $I_1\setminus \{e\}$ are liftable with the edges of $I_2\setminus \{e\}$, we must have $|\delta_G(A_1\cup A_2)|\geq k+2$. But $A_1\cup A_2$ contains all the neighbours of $s$, therefore $|\delta_{G-s}(\overline{A_1\cup A_2})| \geq ((k+2)-\deg(s))$. Thus the right hand side of Equation \ref{intersection of two cuts} is at least $(k-1)+\Big(k-\Big(\frac{\deg(s)-1}{2}\Big)\Big)+\Big(k-\Big(\frac{\deg(s)-1}{2}\Big)\Big)+((k+2)-\deg(s)).$

  This gives a lower bound of $4k-2\deg(s)+2$. On the other hand, the left hand side has the upper bound of $2[(k+1-(\frac{\deg(s)+1}{2}))+(k+1-(\frac{\deg(s)+1}{2}))]= 4k+4-2(\deg(s)+1)=4k-2\deg(s)+2$. Thus, both sides are equal to $4k-2\deg(s)+2$, and the individual upper and lower bounds on each term hold with equality. In particular $|\delta_{G-s}(\overline{A_1\cup A_2})| = ((k+2)-\deg(s))$. The set $\overline{A_1\cup A_2}$ does not contain any neighbours of $s$, therefore, $|\delta_{G}(\overline{A_1\cup A_2})| = |\delta_{G-s}(\overline{A_1\cup A_2})| \newline  = ((k+2)-\deg(s)) <k$, a contradiction since $G$ is $(s,k)$-edge-connected.

  Now, because $\overline{A_1\cup A_2\cup \{s\}}= \emptyset$, the lower bound on the right hand side of Equation \ref{intersection of two cuts} is $(k-1)+(k-(\frac{\deg(s)-1}{2}))+(k-(\frac{\deg(s)-1}{2}))=3k-\deg(s)$. The upper bound on the left hand side of Equation \ref{intersection of two cuts} is $4k-2\deg(s)+2=(3k-\deg(s))+(k-\deg(s)+2)$. This means that, \newline $(|\delta_{G-s}(A_1\cap A_2:\overline{A_1\cup A_2})|+|\delta_{G-s}(A_2\setminus A_1: A_1\setminus A_2)|) \leq (k-\deg(s)+2)/2$.

  We know that $|\delta_{G-s}(A_1\cap A_2:\overline{A_1\cup A_2})|=0$ as $\overline{A_1\cup A_2\cup \{s\}}=\emptyset$. Therefore we have, $|\delta_{G-s}(A_2\setminus A_1: A_1\setminus A_2)| \leq (k-\deg(s)+2)/2$, and this same upper bound also holds in $G$. Since $G$ is $(s,k)$-edge-connected and, for $i=1,2$, $|\delta_G(A_i)|\leq k+1$, then $|\delta_G(A_i)|$ is either $k$ or $k+1$.

  We now show that for $i=1,2$, $|\delta_G(A_i)|=k+1$. If, say, $|\delta_G(A_1)|=k$, then,\newline
  $k=|\delta_G(A_1\setminus A_2 : A_2\setminus A_1)|+|\delta_G(A_1\cap A_2: A_2\setminus A_1)| + |\delta_G(\{s\}:A_1\cap A_2)|+|\delta_G(\{s\}:A_1\setminus A_2)|$. It follows that $|\delta_G(A_2\setminus A_1)|= k-1$, since $s$ has exactly one neighbour in $A_1\cap A_2$ and because $|\delta_G(\{s\}:A_2\setminus A_1)|=|\delta_G(\{s\}:A_1\setminus A_2)|$ (as $A_1$ and $A_2$ correspond to maximal independent sets of the same size in $L(G,s,k)$), a contradiction to the $(s,k)$-edge-connectivity of $G$. The same argument holds for $A_2$.

The equality $|\delta_G(A_1)|=|\delta_G(A_2)|=k+1$ implies $|\delta_{G-s}(A_1)|=|\delta_{G-s}(A_2)|=k+1-(\frac{\deg(s)+1}{2})$, so $|\delta_{G-s}(A_1\setminus A_2: A_2\setminus A_1)|+|\delta_{G-s}(A_1\cap A_2: A_2\setminus A_1)| = |\delta_{G-s}(A_2\setminus A_1: A_1\setminus A_2)|+ \newline |\delta_{G-s}(A_1\cap A_2: A_1\setminus A_2)|= k+1-(\frac{\deg(s)+1}{2})$. So, $|\delta_{G-s}(A_1\cap A_2, A_2\setminus A_1)|=|\delta_{G-s}(A_1\cap A_2, A_1\setminus A_2)|$.

  Then also in $G$, $|\delta_{G}(A_1\cap A_2, A_2\setminus A_1)|=|\delta_{G}(A_1\cap A_2, A_1\setminus A_2)|$. Now since $|\delta_G(A_1\cap A_2)|\geq k$ and $s$ has exactly one neighbour in $A_1\cap A_2$, both $|\delta_G(A_1\cap A_2: A_2\setminus A_1)|$ and $|\delta_G(A_1\cap A_2: A_1\setminus A_2)|$ have to be at least $(k-1)/2$. This lower bound, and the fact that $|\delta(A_1)|= |\delta(A_2)|=k+1$ and each contain $(\deg(s)+1)/2$ neighbours of $s$, imply that $k + 1 \geq \frac{(deg(s) + 1)}{2} + \frac{(k-1)}{2}$, so $\deg(s)\leq k + 2$.\end{proof}

We now present the definitions for boundary-linkedness and a compatible sequence of lifts which were used by Thomassen in \cite{thomassen2016orientations} and Ok, Richter and Thomassen in \cite{ORT2016linkages}. For a set of vertices $C$ (finite or infinite) in a graph $G$, we denote by $G[C]$ the subgraph of $G$ induced by $C$. We write $\delta(C)$ to denote the \emph{boundary} of $C$, which is the set of edges with exactly one end-vertex in $C$. Recall that a \emph{ray} is a one-way infinite path, and an infinite graph $G$ is \emph{locally finite} if the degree of every vertex in it is finite. An \emph{end} of an infinite graph $G$ is an equivalence class of rays, where two rays of $G$ are \emph{equivalent} if there are infinitely many vertex-disjoint paths between them in $G$. An infinite set of vertices $C$ in a graph $G$ is \emph{boundary-linked} (See also \cite[Section 2]{thomassen2016orientations}) if the union of the induced subgraph $G[C]$ together with the boundary edges $\delta(C)$ contains a \emph{witnessing} set $\mathcal R$ of pairwise edge-disjoint rays such that:

\begin{itemize}

\item  the set consisting of the first edge of each ray in $\mathcal R$ is $\delta(C)$, and 

\item the rays of $\mathcal R$ are in one end of $G[C]$.

\end{itemize}

Let $k$ be a positive integer and $G$ a $k$-edge-connected, locally finite graph. Suppose that $C$ is a boundary-linked set of vertices in $G$ with witnessing set of rays $\mathcal{R}$, such that $G-C$ is finite, then by local-finiteness $\delta(C)$ is also finite, and so is $\mathcal{R}$. For every $e\in \delta(C)$, let $R_e$ be the ray in $\mathcal{R}$ containing $e$. For each $e\in \delta(C)$, if $v$ is the end-vertex of $e$ not in $C$, denote the end-vertex in $C$ by $v'$. Note that the boundary-linkedness of $C$ implies that these end-vertices are in one component of $G[C]$. From this it follows, since $G$ is connected and $\delta(C)$ is the only connection between $G-C$ and $G[C]$, that $G[C]$ also is connected. Let $G'$ be the finite graph obtained from $G$ by contracting $C$ to a vertex $c$. We needed $G[C]$ to be connected to be able to do this contraction. Then $G'$ is $k$-edge-connected as $G$ is.  

A sequence of lifts $G^0=G'$, $G^1:=G^{0}_{cx_1,cy_1}$, $G^2:=G^{1}_{cx_2,cy_2}$, \dots, $G^n:=G^{n-1}_{cx_n,cy_n}$ is \emph{$(\mathcal R,k)$-compatible} in $G'$ if for each $i\geq 1$:

\begin{itemize}[topsep=-3pt,itemsep=-3pt]

\item $G^i$ is $(c,k)$-edge-connected (that is, $cx_i$ and $cy_i$ are adjacent in $L(G^{i-1},c,k)$); and 

\item there is a path $P_i$ in $G[C]$ between $x'_i$ and $y'_i$ that is edge-disjoint from each $P_j$ for $j<i$ and from each $R_e$, for every $e\in \delta(C)\setminus \{cx_1,cy_1,\dots,cx_i,cy_i\}$.

\end{itemize}

\vskip 4pt
The \emph{$\mathcal R$-graph} $\mathcal E_{\mathcal R}$ has the edges of $\delta(C)$ as vertices (like $L(G',c,k)$). Two edges $e$ and $f$ of $\delta(C)$ are adjacent in $\mathcal E_{\mathcal R}$ if there are infinitely many vertex-disjoint paths in $G[C]$ having an end-vertex in each of $R_e$ and $R_f$ and otherwise edge-disjoint from all the rays in $\mathcal R$.  Because the finitely many rays in $\mathcal R$ are all in one end, $\mathcal E_\mathcal R$ is connected. The second condition in the above definition of compatibility relates to the adjacency of $cx_i$ and $cy_i$ in $\mathcal E_{\mathcal R}$, with a few more details to be clarified below in the proof of Lemma \ref{lm:compatibleLifting}.

 The following lemma is the main point of this section and a principal part of the proof of Theorem \ref{th:huckInfinite}. We will show later that the technical conditions of this lemma can be satisfied. Note that there could be edges between two of the sets, $C_i$ and $C_j$ for $i\neq j$, in the statement below.

 Note also that, given a finite collection $C_1,C_2,\dots,C_r$ of pairwise disjoint boundary-linked sets, if the graph $G'$ obtained by contracting each $C_i$ to a vertex $c_i$ is finite, then $G'$ can also be regarded as obtained by contracting only one of the sets $C_i$, in the infinite graph $X_i$ obtained by contracting every set in the collection except $C_i$ (note that $X_i-C_i$ is finite). With this view, we can talk about a compatible sequence of lifts in $G'$ according to the definition given above.

\begin{lemma}\label{lm:compatibleLifting} 
Let $k$ and $r$ be positive integers, with $k$ even, and let $G$ be a $k$-edge-connected locally finite graph, and $C_1,C_2,\dots,C_r$ pairwise disjoint sets of vertices such that $G-(\bigcup_{i=1}^rC_i)$ is finite. Suppose also that each $C_i$, $1\leq i \leq r$, is a boundary-linked set with finite boundary and witnessing set $\mathcal R_i$ of rays. Let $H_0$ be the finite graph obtained from $G$ by contracting $C_i$ for every $i\in \{1,2,\dots,r\}$ to a vertex $c_i$. Then, for each $i\in \{1,2,\dots,r\}$, there is a graph $H_i$ that is obtained from $H_{i-1}$ by executing a sequence of lifts at $c_i$ that is $(\mathcal R_i,k)$-compatible in $H_{i-1}$ such that:

\begin{itemize}

\item[(1)] \label{it:evenCompatible} in the case $\deg(c_i)$ is even, the sequence is of length $\deg(c_i)/2$, and $c_i$, which now has degree $0$, is deleted after executing the sequence,

\item[(2)] \label{it:oddCompatible} in the case $\deg(c_i)$ is odd, the sequence is of length $\frac{(\deg(c_i)-(k+1))}{2}$. 

\end{itemize}

Moreover, the graph $H_r$ obtained at the end is $k$-edge-connected and the vertices $c_i$ with odd degree in $H_0$ have degree $k+1$ in $H_r$.

\end{lemma}

\begin{proof}

The graph $H_0$ as defined in the statement is clearly $k$-edge-connected as $G$ is. The proof involves two layers of induction. First, let $i$ be in $\{1,2,\dots,r\}$ and suppose that $H_q$ is defined for all $q<i$.

Then:

\begin{itemize} 

\item for each $q$ such that $1\leq q <i$, either $\deg_{H_0}
(c_q)$ is even, in which case $c_q$ does not exist in $H_{i-1}$, or $\deg_{H_{i-1}}(c_q)= k+1$; and

\item $H_{i-1}$ is $k$-edge-connected.

\end{itemize}

Now we find a sequence of lifts that is $(\mathcal R_i,k)$-compatible in $H_{i-1}$, by iteratively lifting at $c_i$. The reader should be aware that, since edges between different ones of the sets $C_j$ possibly exist, an edge incident with $c_i$ in $H_0$ may have been involved in lifts at some of the $c_q$ with $q<i$, which replaced this edge with another edge incident with $c_i$.  This new edge is still taken to be the initial edge of a ray in $\mathcal R_i$.  This has no impact on the arguments to follow.

Set $H_{i,0}:=H_{i-1}$ and $\mathcal R_{i,0}:=\mathcal R_i$. We define inductively $H_{i,j}$, $\mathcal R_{i,j} \subseteq \mathcal R_i$, and paths $P_j$ (as in the definition of a compatible sequence), for $j=1,2,\dots$ (up to the appropriate upper bound depending on the parity of $\deg(c_i)$). Remember that $L(H_{i-1},c_i,k)$ and $\mathcal E_{\mathcal R_{i}}$ have the same vertex set, which is the set of edges incident with $c_i$ (or $\delta(C_i)$ before contracting $C_i$ to $c_i$). It will be convenient to denote an adjacency in $\mathcal E_{\mathcal R_{i}}$ between, say, $e$ and $e'$ by writing $RR'$ as an adjacency between the rays $R$ and $R'$ in $\mathcal{R}_i$ beginning with $e$ and $e'$, respectively.

For $j\geq 1$, as long as $L(H_{i,j-1},c_i,k)$ and $\mathcal E_{\mathcal R_{i,j-1}}$ have an edge in common, let $R_{j}R'_{j}$ be an edge of both $L(H_{i,j-1},c_i,k)$ and $\mathcal E_{\mathcal R_{i,j-1}}$. By the definition of adjacency in $\mathcal E_{\mathcal R_{i,j-1}}$ there are infinitely many vertex-disjoint paths between $R_j$ and $R'_j$ in $G[C_i]$ that are also edge-disjoint from the other rays in $\mathcal R_{i,j-1}$. Thus there is a path $P_{j}$ in $G[C_i]$ consisting of initial segments of $R_j$ and $R'_j$ (subpaths of the rays beginning from the first vertex which are long enough to have a connection avoiding the finitely many $P_1,\dots,P_{j-1}$) and a path between them that does not have edges in common with any ray in $\mathcal R_{i,j-1}$. In this case we define $\mathcal{R}_{i,j}:= \mathcal R_{i,j-1}\setminus \{R_{j},R'_{j}\}$ (so the next path to define $P_{j+1}$ will be allowed to possibly go through $R_{j}$ and $R'_{j}$ but at a higher level than $P_j$) and define $H_{i,j}$ as the graph obtained from $H_{i,j-1}$ by lifting the pair of edges consisting of the first edge of $R_j$ and the first edge of $R'_j$. Note that since this pair of edges is $k$-liftable, the degree in $H_{i,j}$ of every vertex other than $c_i$ is at least $k$. In particular, each $c_q$ with $q<i$ that is in $H_{i, j-1}$ will have degree $k+1$ in $H_{i,j}$ as well, and the two lifted edges cannot be parallel edges between $c_i$ and $c_q$ for $q <i$ as this will reduce the degree of $c_q$ from $k+1$ to $k-1$. If $\deg_{H_0}(c_i)$ is odd, we stop lifting if the degree gets reduced to $k+1$ regardless of whether we can continue lifting further in this way.

 Suppose the preceding algorithm stops after finding $P_1,\dots,P_{j-1}$. This happens either when $L(H_{i,j-1},c_i,k)$ and $\mathcal E_{\mathcal R_{i,j-1}}$ have no more edges in common, or when $\deg_{H_{i,j-1}}(c_i)$ has reached $k+1$. We show in any case, depending on the parity of $\deg_{H_0}(c_i)$, that  $\deg_{H_{i,j-1}}(c_i)$ is either $0$ or $k+1$, so we can move on to the following vertex $c_{i+1}$ and the construction of $H_{i+1}$. Suppose not for a contradiction, then the alternative is that $\deg_{H_{i,j-1}}(c_i)$ is at least  either $2$ or $k+3$. First consider the case when $\deg_{H_{i,j-1}}(c_i)=2$, then the unique pair of edges incident with $c_i$ is evidently $k$-liftable, and there are infinitely many vertex-disjoint paths between the two rays beginning with this pair that are edge-disjoint from all the previously defined paths (only finitely many). This is a contradiction because it means the algorithm can continue for at least one more step to obtain degree $0$ at $c_i$. Thus, in case $\deg_{H_0}(c_i)$ is even, $\deg_{H_{i,j-1}}(c_i)\geq 4$. If $\deg_{H_0}(c_i)$ is odd, then $\deg_{H_{i,j-1}}(c_i)\geq k+3 \geq 5$. Thus, without loss of generality, we may assume now that $\deg_{H_{i,j-1}}(c_i)\geq 4$.

If $L(H_{i,j-1},c_i,k)$ has a disconnected complement, then the connectedness of $\mathcal E_{R_{i,j-1}}$ implies the that $L(H_{i,j-1},c_i,k)$ and $\mathcal E_{R_{i,j-1}}$ have a common edge, and we can lift one more pair as described above. Therefore, we may assume that $L(H_{i,j-1},c_i,k)$ has a connected complement.  

Since
$\deg_{H_{i,j-1}}(c_i)\geq 4$ and $k$ is even, Lemma \ref{lm:amena} shows in this case that $L(H_{i,j-1},c_i,k)$ consists of an isolated vertex and a balanced complete bipartite graph, whence $\deg_{H_{i,j-1}}(c_i)$ is odd, and so at least $k+3$. On the other hand, Lemma \ref{degree at most k+2} implies the contradiction that $\deg_{H_{i,j-1}}(c_i)\leq k+2$.\end{proof}

We remark here that if we additionally assume that there are no edges between the boundary-linked sets $C_1,\dots,C_r$, then in the case when $\deg(c_i)$ is odd, we can continue the sequence of lifts and reduce the degree of each $c_i$ with odd degree to $3$ while preserving the local edge-connectivity between any two vertices in $G-(\bigcup_{i=1}^rC_i)$ to be at least $k$. In that case, the specific structure for the lifting graph of an isolated vertex plus a balanced complete bipartite graph will be used. For a proof of this fact we refer the reader to a paper by Assem, Koloschin, and Pitz, \cite[Theorem ~ 3.2]{Max2023orientation}, where also it was proved in Theorem 2.2 of the same paper that if the graph is locally finite with at most only countably many ends, then a boundary-linked decomposition exists such that there are no edges between the different boundary-linked sets. For a proof of the case when $r=1$ only, we refer the reader to a paper by Assem \cite[Lemma~3.1]{assem2023towards}. This may be helpful in other situations, but not here, as we need to retain the edge-connectivity of the resulting graph to be at least $k$ (and so $\deg(c_i)$ must be at least $k$). 

\section{linking fan proposition}\label{sec:fan}

The purpose of this section is to prove a fact that will help us deal with the vertices $c_i$ having odd degree in the application of Lemma \ref{lm:compatibleLifting} to prove Theorem \ref{th:huckInfinite} for locally finite graphs. This proposition is a new tool which can be useful in other edge-connectivity related problems in infinite graphs, for example orientations as shown in Section 5. An \emph{initial segment} of a ray is a subpath of it containing its \emph{origin} (that is, the first vertex of the ray). The complement in a ray of an initial segment is a \emph{tail} of the ray.

The proposition shows that if we have $m$ edge-disjoint rays, and we want to avoid a certain finite construction $X$, then as long as $m$ is at most the assumed edge-connectivity, we can find a vertex far enough from $X$, and a fan from that vertex to the $m$ rays consisting of $m$ edge-disjoint paths that are also edge-disjoint from $X$.

This proposition was first presented as \cite[Proposition 3.2.11]{AmenaThesis} in the PhD thesis of the first author. We are grateful to Nathan Bowler for useful discussions regarding the proof, particularly for bringing to our attention an idea from a paper authored by Geelen and Joeris \cite[Lemma 7.1]{BensonJim}, and suggesting its use in the proof.

For a finite set $S$ of vertices and an end $\omega$ in a graph $G$, we denote by $C(S,\omega)$ the unique component of $G-S$ that contains the tails of the rays in $\omega$.

\begin{proposition}(Linking Fan Proposition)\label{linking fan proposition}

Let $k$ be a positive integer, $G$ a $k$-edge-connected locally finite graph, and let $\mathcal{R}$ be a set of pairwise edge-disjoint rays from one end in $G$ such that $|\mathcal{R}|\leq k$. If $X$ is any finite subgraph of $G$ that is edge-disjoint from $\mathcal{R}$, then there is a vertex $v$ and a set of $|\mathcal{R}|$ pairwise edge-disjoint paths from $v$ to $\mathcal{R}$, all edge-disjoint from $X$ and each containing an initial segment of arbitrarily large length of a ray in $\mathcal{R}$.
\end{proposition}

\begin{proof}

Denote $|\mathcal{R}|$ by $m$ and the end containing $\mathcal{R}$ by $\omega$. We inductively define a sequence of pairwise vertex-disjoint finite subgraphs $L_0,L_1,  \dots, L_m$, each connected, except possibly $L_0$.

For each ray $R\in \mathcal R$, let $I_R$ be an initial segment of $R$ of any desired large length. Note that the rays of $\mathcal{R}$ are only edge-disjoint, so it is possible that a vertex after the segment $I_R$ on a ray $R$ is also a vertex on the segment $I_{R'}$ in another ray $R'$. Define $L_0$ as the subgraph of $G$ induced by $V(X) \cup \bigcup_{R\in \mathcal{R}} V(I_R)$, and for each ray $R\in \mathcal R$, let $R_0$ be the initial segment of $R$ from its origin to its last visit to $L_0$ (such a visit exists as $L_0$ consists of a finite number of vertices).

Note also that the segment $R_0$ contains the segment $I_R$ which was of arbitrarily large length. Let $B_0$ be the set of vertices that are on the initial segments $R_0$ for $R\in \mathcal{R}$, and let $C_1:=C(V(L_0) \cup B_0,\omega)$.

Note that $N(V(L_0) \cup B_0)$ (the set of vertices in $G-(V(L_0) \cup B_0)$ adjacent to a vertex in $V(L_0) \cup B_0$) is finite because $V(L_0) \cup B_0$ is finite and $G$ is locally finite. Now define $L_1$ as any finite connected subgraph of $C_1$ containing $N(V(L_0) \cup B_0)\cap V(C_1)$. Then $L_1$ is vertex-disjoint from $X$ (contained in $L_0$) and from the initial segments $R_0$ for $R\in \mathcal R$, and it separates them from $C_1$.

For an integer $n\in \{2,\dots,m\}$, assume that $C_i$ and $L_i$ are defined for all $1\leq i<n$, and that $L_i$ separates $B_0$ and $L_j$ with $j<i$ from $C_i$. Then let $C_n:= C(V(L_{n-1}), \omega)$. Since $L_{n-1}$ is finite and $G$ is locally-finite, $N(L_{n-1})$ too is finite. Let $L_n$ be any finite connected subgraph of $C_n$ that includes $N(L_{n-1}) \cap V(C_n)$. Then $L_n$ separates $L_{n-1}$, and hence also $B_0$ and all $L_j$ for $j<n $, from $C_n$ as $C_n \subseteq C_{n-1} \subseteq \dots \subseteq C_1$ and $B_0$ and $L_0$ are outside $C_1$.

In particular, the initial segments $R_0$ for $R\in \mathcal{R}$ (namely the set $B_0$, which is outside $C_1$) is vertex-disjoint from $L_j$ for all $j>0$ (which are inside $C_1$).

 Fix a vertex $v$ in $C_{m+1} := C(V(L_{m}), \omega)$. Each ray $R$ in $\mathcal R$ has a subpath $P^R$ from its last vertex in $L_0$ (that is the last vertex of $R_0$) to its first vertex in $L_m$. Since $m\leq k$ (the edge-connectivity of $G$), there are $m$ edge-disjoint paths from $v$ to $L_1$. These paths are contained in $C_1$, and consequently vertex-disjoint from $L_0$ (which includes $X$), because $L_0$ is outside $C_1$, while $v$ and $L_1$ are inside it, and $L_1$ separates $L_0$ from $C_1$. Denote these paths by $Q_1,\dots, Q_m$. By construction, each $L_i$, for $i=1,\dots, m$, separates $\bigcup_{j<i} L_j$ from $C_i$, therefore, each of the paths $Q_1,\dots, Q_m$, and $P^R$ for $R\in \mathcal{R}$, has a non-empty (vertex or edge) intersection with each one of the $m$ subgraphs $L_1,\dots, L_m$. Moreover, each of the paths $P^R$ for $R\in \mathcal{R}$, is vertex-disjoint from $X$ except possibly for its end-vertex in $L_0$ (because this is the last visit of $R$ to $L_0 \supseteq X$.)

The following argument is inspired by \cite[Lemma 7.1]{BensonJim} for vertex-disjoint paths. Let $H$ be the graph obtained from the union of $\{v\}$, $\bigcup_{j=0}^{m}L_{j}$, and all the paths $Q_1,\dots, Q_m$, and $P^R$ for $R\in \mathcal{R}$, by contracting $L_0$ to a vertex $u$. Note that both $u$ and $v$ have degree $m$ in $H$, and $H$ is edge-disjoint from $X$ (there could be an edge of $X$ between two vertices in $L_0$ but now these are contracted). 

Suppose for a contradiction that there does not exist $m$ edge-disjoint paths between $v$ and $u$ in $H$. Then by Menger's theorem, there is a set $F$ of less than $m$ edges in $H$ such that $u$ and $v$ are in distinct components of $H-F$. The set $F$ is disjoint from at least one of the paths $Q_1, \dots, Q_m$, one of the paths $P^R$ for $R\in \mathcal{R}$, and one of the subgraphs $L_1,\dots, L_m$. Let these respectively be $Q$, $P$, and $L$, then since $L$ is connected and each of $Q$ and $P$ has a non-empty intersection with it, the union $Q\cup L \cup P$ contains a path between $v$ and $u$ in $H-F$, a contradiction.

 Now the $m$ edge-disjoint paths between $v$ and $u$ in $H$ give us $m$ edge-disjoint paths in $G$ from $v$ to $L_0$ each having as its edge incident with $L_0$ the first edge of $R$ in $P^R$ for a distinct $R\in \mathcal{R}$ (these were the $m$ edges incident with $u$ in $H$ and each one of them is the first edge on a ray $R$ after the segment $R_0$). 
 
 The initial segment $R_0$ of each ray $R\in \mathcal{R}$, from its origin to its last visit to $L_0$, is outside $C_1$ by construction, and edge-disjoint from $H$. Thus, adding these initial segments to the $m$ paths we found gives the desired set, $\mathcal{P}$, of $m$ paths where each origin of a ray of the $m$ rays in $\mathcal{R}$ is the end-vertex (other than $v$) of a distinct one of the $m$ paths in $\mathcal{P}$ (even if two rays have the same origin). Each one of the $m$ initial segments which the paths of $\mathcal{P}$ end in contains, for a distinct $R$, the initial segment $R_0$, which contains the segment $I_R$ of arbitrarily large length, and the first edge after $R_0$ on $R$. \end{proof}

Note that it is possible that a path of $\mathcal{P}$ has edges in common with a ray $R$ in $\mathcal{R}$ (in particular with the path $P^R$ in $H$) before it ends in an initial segment of another ray, say $R' \in \mathcal{R}$, but in any case, the set of paths $\mathcal{P}$ is pairwise edge-disjoint, and is edge-disjoint from $X$.

\section{Generalization of Huck's theorem to infinite graphs}\label{sec:proof}

In this section we prove Theorem \ref{th:huckInfinite}, which shows that Huck's theorem for finite graphs, Theorem \ref{th:huck}, extends to all infinite graphs. The reduction from general infinite graphs to locally finite graphs is adapted from Thomassen \cite{thomassen2016orientations}. Our first step is the reduction from arbitrary infinite graphs to countable graphs. Given a $k$-edge-connected infinite graph $G$, and a finite set $T$ of vertices in it, the following argument of Thomassen shows that $G$ has a countable $k$-edge-connected subgraph $G_{\omega}$ containing $T$. 

Let $G_0:=G[T]$. For each $i\geq 1$, define $G_i$ to be the finite graph obtained from $G_{i-1}$ by taking the union of subgraphs $H_{\{x,y\}}$ over all $\{x,y\}\subseteq V(G_{i-1})$, where $H_{\{x,y\}}$ is a finite subgraph that consists of the union of $k$ edge-disjoint paths between $x$ and $y$ in $G$. The union $G_\omega$ of the $G_i$ evidently contains $T$, is  $k$-edge-connected, and is countable, as required. Thus, it suffices to prove Theorem \ref{th:huckInfinite} for countable graphs. The reduction from countable graphs to locally-finite graphs is more subtle.  Thomassen \cite{thomassen2016orientations} also shows how to do this.

A \emph{splitting} \cite{thomassen2016orientations} of a graph $G$ is a graph $G'$ obtained from $G$ by replacing each vertex $u$ by a set $V_u$ of vertices such that $G'$ has no edge between two vertices in the same $V_u$, and the identification in $G'$ of all vertices of $V_u$ into a single vertex, for each $u \in V(G)$, gives us $G$ back. Note that the edge sets of $G$ and $G'$ are in bijection, in particular the edges with an end-vertex in $V_u$ in $G'$ are precisely the edges incident with $u$ in $G$ after identification.

The following result of Thomassen necessary for the reduction to locally finite graphs is not as simple as the preceding discussion, so we omit its proof which can be read in \cite[Theorem 9]{thomassen2016orientations}. Recall that a \emph{block} is a maximal $2$-vertex-connected subgraph (that is, without a cut-vertex).

\begin{lemma}\cite[Theorem 9]{thomassen2016orientations}\label{countable to locally finite}
Let $k$ be a positive integer, and let $G$ be a countably infinite $k$-edge-connected graph. Then $G$ has a splitting $G'$ that is $k$-edge-connected, and each block of $G'$ is locally finite. \hfill\eop
\end{lemma}

Now let $G$ be a countably infinite $(k+1)$-edge-connected graph, and let $G'$ be the graph obtained by Lemma \ref{countable to locally finite}. Then each block of $G'$ is $(k+1)$-edge-connected (since the blocks of any graph have the same edge-connectivity as the graph). The $k$-linkage problem on $G$ reduces to (possibly smaller) linkage problems on finitely many of the locally finite blocks of $G'$ as follows.

Recall that the \emph{block graph} of $G'$ is the bipartite graph with vertex set the cut-vertices of $G'$ as one side and a vertex for every block of $G'$ as the other side, where there is an edge $aB$ between a cut-vertex $a$ and a block $B$ if and only if $a$ is in the block $B$ in $G'$ \cite[Page 61]{Diestel}. It is not hard to see that the block graph of a connected graph is a tree.

Given a set $T=\{(s_i,t_i):i=1,\dots,k\}$ of $k$ pairs of vertices in $G$, for each $u\in V(T)$ pick any vertex $u'$ from $V_u$ in $G'$. Let us now consider the linkage problem of $\{(s'_i,t'_i):i=1,\dots,k\}$ in $G'$. First, for each $i$, determine a block containing $s'_i$ and a block containing $t'_i$ in $G'$ (a vertex can be in more than one block only if it is a cut-vertex). There is a unique path $P_i$ in the block tree of $G'$ between these two blocks, mark all the blocks that are on this path. Now consider the collection consisting of all the marked blocks for all $i$. For each block $B$ of the blocks in this collection, consider all the values of $i\in\{1,\dots,k\}$ such that $B$ is on the path $P_i$. Now we determine which pairs of vertices we need to find a linkage for in $B$. Note that each $P_i$ is a path in a bipartite graph (tree) that alternately goes between blocks and cut-vertices of $G'$. For each $i$ such that $B$ is on $P_i$, either $B$ is between two cut-vertices, and in this case we take this as a pair of vertices to link, or $B$ is an end-vertex of the path $P_i$ in the block graph, then in this case we take the pair of vertices to be $s'_i$ (or $t'_i$) and the cut-vertex directly following (preceding) $B$ on $P_i$. This gives us a linkage problem of at most $k$ pairs in the locally finite $k$-edge-connected block $B$. By Theorem \ref{th:huckInfinite} which we prove below, such a linkage exists in $B$ if $k$ is odd. These linkages in finitely many blocks of $G'$ together give us a linkage in $G'$ of $\{(s'_i,t'_i):i=1,\dots,k\}$ as follows. For each $i$, we have the following path between $s'_i$ and $t'_i$ consisting of segments from the linkages in the blocks. We have the path $P_i$ in the block graph connecting a block containing $s'_i$ and a block containing $t'_i$ and by construction there is a path from the linkage in the first block on $P_i$ between $s'_i$ and the first cut-vertex on $P_i$, then for any two consecutive cut-vertices on $P_i$ a path connecting them from the linkage in the block between them on $P_i$, and finally a path between the last cut-vertex on $P_i$ and $t'_i$ from the linkage in the last block on $P_i$.

This linkage in $G'$ naturally gives a linkage in $G$, perhaps with edge-disjoint walks rather than paths, as vertex identification preserves edge-disjointness. This completes the reduction from countable graphs to locally finite graphs.

  In addition to the work in the earlier sections, we shall need the following interesting result of Thomassen \cite[Theorem 1]{thomassen2016orientations}. It was also proved in \cite[Theorem 2.2]{Max2023orientation} that if we assume that there are only countably many ends, then we can have a boundary-linked decomposition such that there are no edges between the boundary-linked sets (that is, they coincide with the connected components of $G-A$ for a set $A\supseteq A_0$).

\begin{theorem}\cite[Theorem 1]{thomassen2016orientations} \label{thm boundary linked}
Let $G$ be a connected locally finite graph. If $A_0$ is a vertex set such that the boundary $\delta(A_0)$ is finite, then $V(G)\setminus A_0$ can be partitioned into finitely many pairwise disjoint vertex sets each of which is either a singleton or a boundary-linked set with finite boundary.
\end{theorem}

\textbf{Now for the main contribution of this paper:} The proof of Theorem \ref{th:huckInfinite}.

\begin{proof}
As discussed earlier in this section, we may assume $G$ is locally finite.  Recall that, for this theorem, $k$ is odd and $G$ is $(k+1)$-edge-connected. Let $T$ denote the given set of $k$ pairs of vertices and let $A$ be the set of vertices that appear in those pairs. Since $G$ is locally finite, $\delta(A)$ is finite. By Theorem \ref{thm boundary linked}, $V(G)\setminus A$ can be partitioned into finitely many pairwise disjoint vertex sets that are either singletons or boundary-linked sets with finite boundary. 

Adding the singletons to $A$ yields a finite set $A'$ containing the vertices of the pairs in $T$ such that $V(G)\setminus A'$ is partitioned into finitely many pairwise disjoint sets $C_1,\dots, C_r$ such that each $C_i$ is boundary-linked with finite boundary.  For each $i=1,2,\dots,r$, there is a set $\mathcal R_i$ of pairwise edge-disjoint rays, all in the same end of $G[C_i]$ such that the set consisting of the first edge of each ray is exactly $\delta(C_i)$, but otherwise the rays are contained in $G[C_i]$.

Let $G'$ be the graph obtained from $G$ by contracting each $C_i$ to a single vertex $c_i$. The idea is to find a linkage of $T$ in the finite graph $G'$ such that the paths of the linkage that go through $c_1,\dots,c_r$ are replaceable with actual paths in $G$. 

Applying Lemma \ref{lm:compatibleLifting} with the even connectivity $k+1$, there exists a sequence of lifts, consisting of lifts at $c_1, c_2, \dots,c_r$ in order,
such that for every $i\in \{1,\dots,r\}$, the part of the sequence performed at $c_i$ is $(\mathcal{R}_i,k+1)$-compatible in the graph obtained from $G'$ after doing the lifts at $c_j$ for all $j<i$ in order. There are $\deg_{G'}(c_i)/2$ lifts done at $c_i$, if $\deg_{G'}(c_i)$ is even, and $(\deg_{G'}(c_i)-(k+2))/2$ lifts done at $c_i$, if $\deg_{G'}(c_i)$ is odd. Moreover, the finite graph $H$ that is the result of performing all these lifts is also $(k+1)$-edge-connected as $G$ is, and,
of the vertices $c_1, \dots, c_r$ only those with odd degree in $G'$ are contained in $H$ where they have degree $k+2$.

As $T$ is in $A' \subseteq V(H)$, Huck's Theorem \ref{th:huck} shows that $H$ has a weak $T$-linkage $\{P_1,P_2,\dots,P_k\}$.  

 To turn this into a linkage in $G$, first replace each edge $e$ of the linkage in $H$ that has arisen by lifts at various $c_i$ by a path in $G$ as follows. Beginning with $e$ and the two edges it directly resulted from by lifting, iteratively in order, for each $c_i$ contributing a lift of $e_i$ and $e'_i$ towards the formation of $e$, replace this lift with the path consisting of $e_i$ and $e'_i$ connected by the $e_ie'_i$-path in $G[C_i]$ which resulted from compatible lifting according to Lemma \ref{lm:compatibleLifting}. 
 

 Note that an edge in $H$ which resulted from lifting does not necessarily have its two end-vertices in $A$. This is because an edge with one end-vertex in $C_i$ and one end-vertex in $C_j$ for $i\neq j$ possibly exists. So it is possible that an edge $e$ in $H$ incident with $c_i$ is the result of lifting two edges at $c_j$ for $j\neq i$, where one of the two lifted edges is an edge between $c_i$ and $c_j$. There is a unique edge $e'$ in the replacement path for $e$ that is incident with $c_i$ as $e$ is. But $e'$ is an edge in $G$ (in $\delta(C_i)$).

Next, observe that if some $P_j$ contains an edge incident in $H$ with a $c_i$, then $\deg_{P_j}(c_i)=2$ (recall that the end-vertices of $P_j$ are in $A'$, so different from $c_i$). Thus, the set $E_i$  of edges in $\bigcup_{j=1}^k P_j$ incident in $H$ with $c_i$, is of even size, however, $\deg_{H}(c_i)=k+2$ is odd, so $|E_i|\le k+1$. Let $E'_i:= \{e': e\in E_i\}$ where $e'$ is the unique edge in $\delta(C_i)$ corresponding to $e$ as described in the previous paragraph in case $e$ resulted from lifts, and $e'$ is $e$ otherwise.

 Let $\mathcal R'_i$ consist of the at most $k+1$ rays in $\mathcal R_i$ that have their initial edges in $E'_i$.  By the definition of a compatible lifting, and since the edges of $E'_i$ were not lifted, each ray in $\mathcal R'_i$ is edge-disjoint from all of the paths constructed in the compatible liftings at $c_i$.  

 Let $X_i$ be the set of edges in $C_i$ occurring in all those paths which resulted from compatible lifting. Since $k+1$ is the assumed edge-connectivity of $G$ and $|E'_i|\le k+1$, we can apply Proposition \ref{linking fan proposition} to $\mathcal R'_i$. Thus, there is a vertex $v_i$ in $C_i$ and a set $\mathcal Q_i$ of $|E'_i|$ pairwise edge-disjoint paths in $\mathcal C_i$ from $v_i$ to the initial edges of the rays in $\mathcal R'_i$ that are edge-disjoint from $X_i$.  

Suppose $e,f$ are both incident with $c_i$ in $H$ and are consecutive in a path $P_j$ for $j\in \{1,\dots,k\}$ (so $e,f \in E_i$). Then there are rays $R_e$ and $R_f$ in $\mathcal R'_i$ containing $e$ and $f$ and corresponding paths $Q_e$ and $Q_f$ in $\mathcal Q_i$ containing non-trivial initial segments of $R_e$ and $R_f$.  In $G$ we connect $e$ and $f$ using $Q_e\cup Q_f$, completing the $T$-linkage in $G$. \end{proof}

\section{Orientations of infinite graphs}\label{sec:orientations}

In this section, we demonstrate that Proposition \ref{linking fan proposition} is also useful for proving the existence of highly connected orientations in infinite graphs. We prove in Theorem \ref{orientation theorem} that $(2k+1)$-edge-connected infinite graphs admit a $k$-arc-connected orientation. This brings us closer to the conjecture of Nash-Williams that $2k$-edge-connectivity suffices for infinite graphs as it does for finite graphs \cite{nash1960orientations}.

   Recall that an \emph{orientation} of a graph $G$ is obtained by replacing each edge $uv$ of $G$ with an oriented arc, either $(u,v)$ or $(v,u)$, and that an orientation, or the resulting directed graph, is $k$-\emph{arc-connected} if for any two vertices $x$ and $y$ there are $k$ arc-disjoint paths between them directed from $x$ to $y$ (and $k$ such paths directed from $y$ to $x$).

   In 1960 Nash-Williams proved for finite graphs that every $2k$-edge-connected graph admits a $k$-arc-connected orientation \cite{nash1960orientations}. He conjectured that the same is true for infinite graphs. After more than $50$ years Thomassen showed in 2016 that an edge connectivity of $8k$ suffices for infinite graphs to have a $k$-arc-connected orientation \cite{thomassen2016orientations}. Then in 2023, in \cite{assem2023towards}, Assem proved that for $1$-ended locally finite graphs, an edge connectivity of $4k$ is enough. Pitz remarked that, in fact, $4k$-edge-connectivity is a sufficient condition for all infinite graphs, and that this can be shown by combining an observation on the extension of orientations in finite graphs from Eulerian (open or closed) subgraphs, with Thomassen's proof. This remark appears in the foreword of \cite[Section 6]{Max2023orientation} where in the same paper Assem, Koloschin, and Pitz showed that the conjecture (of $2k$) is true for locally finite graphs with countably many ends, still using techniques that in outline are guided by Thomassen's approach.  

We prove, again following Thomassen’s general approach, that $(2k+1)$-edge-connectivity is sufficient for all infinite graphs. To do this, we first present the following Theorem \ref{immersion result} where Proposition \ref{linking fan proposition} is used in finding a highly connected immersion. The statement of the theorem contains a detailed description of the immersion graph, but for a more compact statement we also have Corollary \ref{immersion corollary} below.

This immersion result is similar in nature to \cite[Theorem 4]{thomassen2016orientations}, \cite[Theorem 4.2]{assem2023towards}, and \cite[Theorem 3.2]{Max2023orientation}. The good thing about this immersion theorem here in comparison to those theorems is that it avoids a different restriction of each one of them as follows. The connectivity of the immersion is only $1$ less than the connectivity of the graph, unlike \cite[Theorem 4]{thomassen2016orientations} where the connectivity is reduced by a factor of $2$. There is no restriction on the number of ends, unlike \cite[Theorem 4.2]{assem2023towards} where the graph is assumed to be $1$-ended and unlike \cite[Theorem 3.2]{Max2023orientation} where it is assumed to have only countably many ends. Finally, there is no assumption preventing the presence of edges between the sets of a boundary-linked decomposition, unlike \cite[Theorem 3.2]{Max2023orientation} where those sets are assumed to coincide with the connected components.

Recall that an \emph{immersion} of a graph $H$ in a graph $G$ is a subgraph $H'$ of $G$ isomorphic to a graph obtained from $H$ by replacing each edge of $H$ with a path between its end-vertices (all of whose vertices other than the two ends are new vertices that are not in $H$) such that these paths are pairwise edge-disjoint. The subset of vertices of $H'$ in bijection with the vertices of $H$ are called the \emph{branch vertices} of the immersion $H'$.

From Thomassen's Theorem \ref{thm boundary linked}, it follows that given any finite set of vertices $A_0$ in a connected locally finite graph $G$, then $V(G)\setminus A_0$ can be partitioned into finitely many pairwise disjoint vertex sets each of which is either a singleton or a boundary-linked set with finite boundary. Adding the singletons to $A_0$, we get a finite set $A$ such that $V(G)\setminus A$ is partitioned into boundary-linked sets with the aforementioned properties. There could be edges between two different boundary-linked sets in this decomposition. If the size of the boundary of a set is even (odd), we will for simplicity say that the set is of even (odd) boundary respectively.

\begin{theorem}
\label{immersion result}
Let $k$ be a positive integer and $G$ a $(2k+1)$-edge-connected, locally finite graph, and let $A$ be a finite set of vertices in $G$ such $V(G)\setminus A$ is partitioned into finitely many pairwise disjoint boundary-linked sets each of finite boundary. Then $G$ contains an immersion $H'$ of a finite $2k$-edge-connected graph $H$ with the following properties.

\begin{itemize}

\item[(i)] $G[A]$ is a subgraph of both $H$ and $H'$ and is its own image under immersion, that is, the path in $H'$ replacing an edge of $G[A]$ is the edge itself.

\item[(ii)] The vertices of $V(H)\setminus A$ are in bijection with the boundary-linked sets of odd boundary such that each vertex is mapped under immersion in $H'$ to a vertex $x_C$ in its corresponding boundary-linked set $C$ of odd boundary. Each vertex of $V(H)\setminus A$ is of degree $2k+1$ and the edges incident with it are mapped in $H'$ to $2k+1$ edge-disjoint paths from $x_C$ to $2k+1$ distinct edges in the boundary of $C$. The edges of these paths are contained in $C$ except for the last edge in the boundary.

\item[(iii)] For every boundary-linked set $C$, all the edges of the boundary $\delta_{G}(C)$ are contained in $H'$.

\end{itemize}

\end{theorem}

\begin{proof}
 Let $C_1,\dots,C_r$ be the given boundary-linked decomposition of $V(G)\setminus A$, and let $\mathcal{R}_i$ for $i=1,\dots,r$ be a witnessing set of rays. Let $G^*$ be the graph obtained from $G$ by contracting every $C_i$ to a vertex $c_i$ for $i=1,\dots,r$.
 
 By Lemma \ref{lm:compatibleLifting} applied to the even connectivity $2k$, there exists a sequence of lifts, consisting of lifts at $c_1, c_2, \dots,c_r$ in order, such that for every $i\in \{1,\dots,r\}$, the part of the sequence performed at $c_i$ is $(\mathcal{R}_i, 2k)$-compatible in the graph obtained from $G^*$ after doing the lifts at $c_j$ for all $j<i$ in order. The finite graph $H$ that is the result of performing all these lifts is $2k$-edge-connected, and only the vertices $c_i$ of odd degree in $G^*$ are contained in $H$ where they have degree $2k+1$.

Since $G$ is $(2k+1)$-edge-connected, then by Proposition \ref{linking fan proposition}, in every $C_i$ of odd boundary there is a vertex $x_i$ and $2k+1$ pairwise edge-disjoint paths contained in $G[C_i]$ (except for their last edge), each from $x_i$ to a distinct one of the edges of $\delta(C_i)$ corresponding to the $2k+1$ edges incident with $c_i$ in $H$ (each of which is the first edge of a ray from the witnessing set). Moreover, these paths are edge-disjoint from the, in total finite, subgraph consisting of all the linking paths in $G[C_i]$ (obtained by successive $(\mathcal{R}_i, 2k)$-compatible lifting) which connect the lifted pairs of edges from $\delta(C_i)$. Note that an edge incident with $c_i$ in $H$ is not necessarily an edge in $G$, however, it corresponds to an edge in $\delta(C_i)$ (possibly with a different end-vertex outside $C_i$). 

Let $X:=\{x_i: |\delta(C_i)| \; \text{is odd}\}$. Then $G$ contains an immersion of the $2k$-edge-connected graph $H$ whose set of branch vertices is $A\cup X$. The immersion $H'$ is constructed from $H$ by applying the following two stages of replacing edges by paths:

\begin{itemize}

\item[(1)] First, any edge $e=uv$ in $H$ that resulted from (possibly several) lifts is recursively replaced by paths until we obtain a path in $G^*$ between its end-vertices. We start with $e$ as a path of length $1$ between $u$ and $v$, then at any step of the recursion we replace every edge of the current path connecting $u$ and $v$ that resulted from lifting with the two edges it was directly obtained from.

This gives us a longer path at each step, and once the iteration is complete, we get a path $P$ in $G^*$ between $u$ and $v$. Now, any two consecutive edges in $P$ meet at one of the vertices $c_i$ in $G^*$. We insert in $P$ the linking path between them in $G[C_i]$ which was obtained by $(\mathcal{R}_i, 2k)$-compatible lifting according to Lemma \ref{lm:compatibleLifting}. In this way we obtain from $P$ a path in $G$.

Note that an edge obtained by lifting, such as $e$ above, may have both end-vertices in $A$ or may be incident with one of the vertices $c_i$, and in both cases the treatment in the same.

\item[(2)] The vertices $x_i$ are the branch vertices in bijection with the vertices $c_i$. After applying (1), the $2k+1$ edges incident with $c_i$ are now replaced by the corresponding original edges from $\delta(C_i)$ incident with $c_i$ in $G^*$. We then replace these with the $2k+1$ edge-disjoint paths emanating from $x_i$ which we found above, and which end in these edges.

\end{itemize}

It is clear from the construction that $(ii)$ is satisfied. It can also be easily seen from the construction that any edge of $G$ with both end-vertices in $A$ is contained in both $H$ and in its immersion $H'$ in $G$ as its own image since such an edge is not affected at all by lifting because it is not incident with any of the vertices $c_i$. This shows that $(i)$ holds.

For every boundary-linked set $C_i$, any edge in $\delta(C_i)$ is either lifted with another edge, or is one of the $2k+1$  edges remaining in case $|\delta(C_i)|$ is odd. In the first case, it follows from (1) that this edge is contained in one of the paths of $H'$ obtained by the recursive replacements. In the second case, it follows from (2) that this edge is contained in one of the $2k+1$ paths connecting $x_i$ to the boundary of $C_i$, and so is in $H'$ as well. This proves $(iii)$.\end{proof}

The above discussion, about boundary-linked sets, before the theorem directly gives us the following corollary.

\begin{corollary}\label{immersion corollary}
Let $k$ be a positive integer and $G$ a $(2k+1)$-edge-connected, locally finite graph, and let $A_0$ be a finite set of vertices in $G$. Then $G$ contains an immersion $H'$ of a finite $2k$-edge-connected graph $H$ such that $G[A_0]$ is a subgraph of both $H$ and $H'$ and is its own image under immersion. In particular, $A_0$ is a subset of the branch vertices of $H'$.
\end{corollary}

Now we prove the new orientation result. We thank Max Pitz for many helpful conversations about the following proof, and for pointing out that the above immersion result can be used to prove the sufficiency of $(2k+1)$-edge-connectivity for $k$-arc-connected orientations, which is a better result than what we originally thought we can get. Recall that an orientation of a graph $G$ is \emph{well-balanced} if for any two vertices $x$ and $y$ in $G$ the maximum number of edge-disjoint directed paths from $x$ to $y$ in the orientation is at least half the maximum number of edge-disjoint paths between them in $G$, rounded down.

\begin{theorem}\label{orientation theorem}
Let $k$ be a positive integer and $G$ a $(2k+1)$-edge-connected infinite graph. Then $G$ has a $k$-arc-connected orientation.
\end{theorem}
\begin{proof}
By the arguments in Sections 7 and 8 of Thomassen's paper \cite{thomassen2016orientations}, it suffices to prove this for locally finite graphs. So, let $v_0, v_1, \dots$ be an enumeration of the vertices in $G$. We define a sequence of finite directed graphs $\vec{W}_0 \subseteq\vec{W}_1 \subseteq \vec{W}_2 \subseteq \dots$ whose underlying graphs $W_1 \subseteq W_2 \subseteq \dots$ are subgraphs of $G$ with the following properties:

\begin{itemize}

\item[(1)] for every $n\geq 0$, $v_n \in V(W_n)$,

\item[(2)] for every $n\geq 1$,

\begin{itemize}

\item[(i)] each $W_n$ is an immersion in $G$ of a $2k$-edge-connected finite graph $H_n$ such that both $W_n$ and $H_n$ contain $W_{n-1}$ which is its own image under immersion,

\item[(ii)] each path of the immersion $W_n$ replacing an edge of $H_n$ between two branch vertices has all of its edges directed in one direction in $\vec{W}_n$  (from one branch vertex toward the other), and

\item[(iii)]for any two branch vertices $y$ and $z$ of $W_n$ there are $k$ arc-disjoint paths directed from $y$ to $z$ in $\vec{W}_n$

\end{itemize}

\end{itemize}

Note that $W_0$ is not required to satisfy (2), $W_n$ is not necessarily $2k$-edge-connected, and $\vec{W}_n$ is not necessarily $k$-arc-connected, but these are the local connectivities between their branch vertices. Let $W_0$ be the graph consisting of $v_0$ only.

Assume that $\vec{W}_n$ is defined, and that either it is a single vertex with no edges (the case of $W_0$), or it satisfies the properties (1) and (2). Let $v$ be the first vertex in the above enumeration that is not in $W_n$, so $v=v_i$ for some $i\geq n+1$. Let $A$ be the finite set consisting of $v$, $V(W_n)$, and the finitely many singletons in a boundary-linked decomposition of $V(G)\setminus (V(W_n)\cup \{v\})$ according to Theorem \ref{thm boundary linked}. Then by Corollary \ref{immersion corollary}, $G$ contains an immersion $H'$ of a finite $2k$-edge-connected graph $H$ such that $G[A]$ is contained in both $H$ and $H'$ as its own image under immersion. In particular $A$ is a subset of the branch vertcies of $H'$.

By $(2)$, $W_n$ is an immersion. Let $W^*$ be the graph obtained from $W_n$ by identifying its branch vertices into one vertex $w^*$. Then $W^*$ is an Eulerian graph because it consists of edge-disjoint cycles going through $w^*$, or it consists of only one vertex without edges in case $n=0$. These cycles come from the paths (replacing edges) of the immersion whose end-vertices are now identified. The orientation $\vec{W}^*$ of the Eulerian graph $W^*$ naturally inherited from $\vec{W}_n$ by this identification is consistent (i.e. along an Eulerian tour) because each path of the immersion $W_n$ (which is now a cycle through $w^*$) is directed in one direction by assumption (ii) of the induction hypothesis. In case $n=0$, it is vacuously true that the orientation is consistent as there are no edges to orient.  

Let $H^*$ be the graph obtained from $H$ as a result of applying the aforementioned identification in it (recall that $W_n\subseteq G[A] \subseteq H$).

The oriented part of $H^*$, that is $\vec{W^*}$, is a consistently oriented Eulerian subgraph, so by \cite[Corollary 2]{kiraly2006simultaneous}, the orientation $\vec{W^*}$ is extendable to a well-balanced orientation $\vec{H}^*$ of $H^*$. Since $H$ is $2k$-edge-connected, then $H^*$ too is, and so $\vec{H}^*$ is $k$-arc-connected. The orientations $\vec{W}_n$ and $\vec{H}^*$ together give an orientation $\vec{H}$ of $H$. In $H'$ give each path between two branch vertices the same direction as the edge of $H$ it corresponds to. This defines an orientation $\vec{H'}$. We then define $\vec{W}_{n+1}$ to be the oriented graph $\vec{H'}$.

Then $\vec{W}_n \subseteq \vec{W}_{n+1}$, and it is obvious from the construction that $v_{n+1}\in V(W_{n+1})$, so (1) holds. Since $H'$ (=$W_{n+1}$), which is the underlying graph of $\vec{W}_{n+1}$, is an immersion of a $2k$-edge-connected graph $H$ and both contain $W_n$, then (i) is satisfied. The way the orientation of $H'$ was defined from the orientation of $H$ guarantees (ii). It remains to show that (iii) too holds.

We first show that the orientation $\vec{H}$ is $k$-arc-connected as $\vec{H^*}$ is. Once this is proved, it follows directly that for any two branch vertices $y$ and $z$ of $H' (=W_{n+1})$, there are $k$ arc-disjoint paths directed from $y$ to $z$ in $\vec{W}_{n+1}$ (that is (iii) is satisfied) obtained from the $k$ arc-disjoint directed paths between their pre-images in $H$ by replacing edges with paths.

 By Menger's Theorem, we need to show that for any edge-cut in $H$ with sides $Y$ and $Z$, there are at least $k$ edges directed from $Y$ to $Z$, and at least $k$ edges directed from $Z$ to $Y$ in $\vec{H}$. Recall that $\vec{H}$ is the union of $\vec{W}_n$ and $\vec{H}^*$. We consider two cases, whether or not this cut separates two branch vertices of $\vec{W}_n$. If it does, then the cut contains the desired $k$ arcs in each direction by (iii) for $\vec{W}_n$. So suppose all the branch vertices of $W_n$ lie on one side of the cut, say $Y$, then by identifying all these branch vertices to $w^*$, the edges of this cut still form a cut, but in $H^*$ with $w^*\in Y$. Since $\vec{H}^*$ is $k$-arc-connected, it contains $k$ arcs from $Y$ to $Z$ and $k$ arcs from $Z$ to $Y$, and these are also arcs in $\vec{H}$ between $Y$ and $Z$, except that some of them may have as one end-vertex $w^*$ in $H^*$ but have in place of it a branch vertex of $W_n$ in $H$.

 The union of the oriented graphs $W_n$ defines an orientation of $G$: $\vec{G}= \bigcup_{n\geq 0} W_n$. This orientation is $k$-arc-connected as follows. By construction, for any two vertices
$y$ and $z$ of $G$, there exist $n\geq 1$ such that $y$ and $z$ are in $W_n$. The set $V(W_n)$ is contained in $W_{n+1}$ as a subset of the branch vertices, so by (iii) for $\vec{W}_{n+1}$ there are $k$ arc-disjoint paths directed from $y$ to $z$ in $\vec{W}_{n+1}$, and these are also paths in $\vec{G}$.\end{proof}

\bibliographystyle{plain}
\bibliography{reference}

\end{document}